\documentclass[11pt, amsfonts]{amsart}

\usepackage{amsmath,amssymb,amsthm,here}
\usepackage[all]{xy}
\usepackage[dvipdfm,colorlinks=true]{hyperref}

\numberwithin{equation}{section}

\SelectTips{eu}{12}

\newtheorem{theorem}{Theorem}[section]
\newtheorem{proposition}[theorem]{Proposition}
\newtheorem{lemma}[theorem]{Lemma}
\newtheorem{corollary}[theorem]{Corollary}

\theoremstyle{definition}

\newtheorem{question}[theorem]{Question}

\theoremstyle{remark}

\newcommand{\Z}{\mathbb{Z}}

\title[Higher homotopy associativity in the Harris decomposition]{Higher homotopy associativity in the Harris decomposition of Lie groups}

\author{Daisuke Kishimoto}
\address{Department of Mathematics, Kyoto University, Kyoto, 606-8502, Japan}
\email{kishi@math.kyoto-u.ac.jp}

\author{Toshiyuki Miyauchi}
\address{Department of Applied Mathematics, Faculty of Science, Fukuoka University, Fukuoka, 814-0180, Japan}
\email{miyauchi@math.sci.fukuoka-u.ac.jp}


\begin{document}

\baselineskip.525cm

\maketitle

\begin{abstract}
  For certain pairs of Lie groups $(G,H)$ and primes $p$, Harris showed a relation of the $p$-localized homotopy groups of $G$ and $H$. This is reinterpreted as a $p$-local homotopy equivalence $G\simeq_{(p)}H\times G/H$, and so there is a projection $G_{(p)}\to H_{(p)}$. We show how much this projection preserves the higher homotopy associativity.
\end{abstract}


\section{Introduction}

Lie groups decompose into products of small spaces when localized at a prime $p$. This is called the mod $p$ decomposition and is fundamental in the homotopy theory of Lie groups. Then it is important to study relations between the mod $p$ decomposition and the group structures (or the loop space structures) of Lie groups, and there are several results on such relations \cite{HK,HKMO,HKO,HKST,HKT,KK,K,KT,R,Th}. In this paper, we study a relation between the group structures (or the loop structures) of Lie groups and maps between Lie groups arising from the classical mod $p$ decomposition due to Harris \cite{Ha1,Ha2}.

Let $(G,H)=(SU(2n+1),SO(2n+1)),\,(SU(2n),Sp(n)),\,(SO(2n),SO(2n-1)),\,(E_6,F_4)$, $(Spin(8),G_2)$ and let $p$ be any prime $\ge 5$ for $(G,H)=(E_6,F_4)$, any prime $\ne 3$ for $(G,H)=(Spin(8),G_2)$, and any odd prime otherwise. Then $H$ is a subgroup of $G$ in the obvious way so that there is a fibration $H\to G\to G/H$. Harris \cite{Ha1,Ha2} showed that the associated homotopy exact sequence splits after localizing at a prime $p$ such that
\begin{equation}
  \label{Harris homotopy}
  \pi_*(G)_{(p)}\cong\pi_*(H)_{(p)}\oplus\pi_*(G/H)_{(p)}.
\end{equation}
The proof of Harris actually implies a stronger result such that there is a $p$-local homotopy equivalence
\begin{equation}
  \label{Harris}
  G\simeq_{(p)}H\times G/H
\end{equation}
which is one of the most classical mod $p$ decomposition of Lie groups. In particular, there is a projection $G_{(p)}\to H_{(p)}$ which has been treated only as a continuous map so far. However, we are interested in relations between the projection and the group structures (or the loop space structures) of $G$ and $H$, and so we naively ask how much this projection respects the group structures of $G$ and $H$. We make this naive question more precise. Stasheff \cite{St} defined $A_k$-maps for $k\ge 2$ between loop spaces as H-maps preserving the $(k-2)^\text{th}$ higher homotopy associativity. Then $A_k$-maps form a gradation between continuous maps and loop maps , and so our question is made precise as follows.

\begin{question}
  Given a prime $p$, for which $k$ is the projection $G_{(p)}\to H_{(p)}$ an $A_k$-map?
\end{question}

The aim of this paper is to answer this question. Since the case $(G,H)=(SU(2),Sp(1))$ is trivial, we will exclude it throughout.

\begin{theorem}
  \label{main}
  Let $(G,H),\,a_k$ and $p$ be as in the table below. Then for $k\ge 2$ the following statements hold:
  \begin{enumerate}
    \item for $(G,H)\ne(SO(2n),SO(2n-1))$ the projection $G_{(p)}\to H_{(p)}$ is an $A_k$-map if and only if $p\ge a_k$;
    \item for $(G,H)=(SO(2n),SO(2n-1))$
    \begin{enumerate}
      \item if $p\ge a_k$ then the projection  $G_{(p)}\to H_{(p)}$ is an $A_k$-map;
      \item if $p<a_k-n+2$ then the projection $G_{(p)}\to H_{(p)}$ is not an $A_k$-map.
    \end{enumerate}
  \end{enumerate}

  \renewcommand{\arraystretch}{1.2}

  \begin{figure}[H]
    \centering
    \begin{tabular}{c|ccc}
      \hline
      $(G,H)$&$(SU(2n+1),SO(2n+1))$&$(SU(2n),Sp(n))$&$(SO(2n),SO(2n-1))$\\
      $a_k$&$k(2n+1)$&$2kn-1$&$2(k-1)(n-1)+n$\\
      $p$&$p\ge 3$&$p\ge 3$&$p\ge 3$\\\hline
      $(G,H)$&$(E_6,F_4)$&$(Spin(8),G_2)$&\\
      $a_k$&$12k-5$&$6k-2$&\\
      $p$&$p\ge 5$&$p\ne 3$&\\\hline
    \end{tabular}
  \end{figure}
\end{theorem}

Harris \cite{Ha2} showed the decomposition \eqref{Harris homotopy} by constructing a specific map $G/H\to G$, which is a $p$-local section of the projection $G\to G/H$, from a finite order self-map of $G$. But if we use the mod $p$ decomposition of Mimura, Nishida, and Toda \cite{MNT} and its naturality instead, then we get the decomposition \eqref{Harris} for more pairs of Lie groups, e.g. $(SU(n),SU(n-1))$ for $p\ge n$. Our method for showing the projection $G_{(p)}\to H_{(p)}$ is an $A_k$-map does not use a specific map $G/H\to G$, and so we get the following general result. We set notation to state it. Suppose that a connected Lie group $G$ has no $p$-torsion in the integral homology for an odd prime $p$. Then its mod $p$ cohomology is an exterior algebra generated by odd degree elements. Suppose further that a pair of connected Lie groups $(G,H)$ admits the decomposition \eqref{Harris}. Then the mod $p$ cohomology of $H$ and $G/H$ are also exterior algebras generated by odd degree elements. Let $2m-1$ and $2l-1$ be the largest dimensions of the mod $p$ cohomology generators of $H$ and $G/H$, respectively. We define
$$b_k=\max\{(k-1)m+l,kl\}.$$
Notice that if $(G,H)$ is as in Theorem \ref{main} then it satisfies the above conditions. The following table gives a list of $(m,l)$ for $(G,H)$ in Theorem \ref{main}.

\renewcommand{\arraystretch}{1.2}

\begin{figure}[H]
  \centering
  \begin{tabular}{c|ccc}
    \hline
    $(G,H)$&$(SU(2n+1),SO(2n+1))$&$(SU(2n),Sp(n))$&$(SO(2n),SO(2n-1))$\\
    $(m,l)$&$(2n+1,2n+1)$&$(2n,2n-1)$&$(2n-2,n)$\\\hline
    $(G,H)$&$(E_6,F_4)$&$(Spin(8),G_2)$&\\
    $(m,l)$&$(12,9)$&$(6,4)$&\\\hline
  \end{tabular}
\end{figure}

\noindent It holds that $a_k=b_k$ for $(G,H)$ in Theorem \ref{main} except for $(G,H)=(E_6,F_4)$, in which case $a_k=b_k-2$.

\begin{theorem}
  \label{main general}
  Let $(G,H)$ be a pair of connected Lie groups, and let $p$ be an odd prime. Suppose $G$ has no $p$-torsion in the integral homology and the decomposition \eqref{Harris} holds. Then for $p\ge b_k$ the projection $G_{(p)}\to H_{(p)}$ is an $A_k$-map.
\end{theorem}

Theorem \ref{main general} is a consequence of the following stronger statement (Proposition \ref{regular A_k refined}).

\begin{theorem}
  \label{main equiv}
  Under the condition of Theorem \ref{main general}, there is an $A_k$-structure on $(G/H)_{(p)}$ such that the decomposition
  $$G\simeq_{(p)} H\times G/H$$
  is an $A_k$-equivalence.
\end{theorem}

The main technique for proving Theorem \ref{main general} is a refinement of the reduction of the projective spaces of $p$-regular Lie groups established in the paper \cite{HKT} on the higher homotopy commutativity of localized Lie groups. Then, indirectly though, our result is connected to higher homotopy commutativity, Sugawara and Williams $C_k$-spaces, where we refer to \cite{HKT} for their definitions. For example, we have the following.

\begin{corollary}
  \label{C_n}
  Let $p$ be an odd prime. The following are equivalent:
  \begin{enumerate}
    \item the projection $SU(2n+1)_{(p)}\to SO(2n+1)_{(p)}$ is an $A_k$-map;
    \item $SU(2n+1)_{(p)}$ is a Sugawara $C_k$-space;
    \item $SU(2n+1)_{(p)}$ is a Williams $C_k$-space.
  \end{enumerate}
\end{corollary}

It would be interesting to find a direct connection between an $A_k$-structure of the projection $G_{(p)}\to H_{(p)}$ and the higher homotopy commutativity of $G_{(p)}$ and $H_{(p)}$.

\textit{Acknowledgement:} We are grateful to Mitsunobu Tsutaya for useful advices on $A_n$-structures and to the referee for pointing out an ambiguity in the calculation of a Samelson product in $Spin(8)$ in the earlier version and for the suggestion to add Theorem \ref{main equiv}. The first author was supported in part by JSPS KAKENHI (No. 17K05248).


\section{Projective spaces for products}

Throughout this section, we assume that spaces are path-connected. This section recalls  the reduction of the projective space of a product of $A_n$-spaces established in \cite{HKT} and shows its property that we will use. We refer to \cite{St,IM,HKT} for the basics of $A_n$-spaces, their projective spaces, and $A_n$-maps. Let $X$ be an $A_n$-space and $P^kX$ be the $k^\text{th}$ projective space of $X$ for $k\le n$. If $X$ is an $A_\infty$-space then we write the canonical map $P^kX\to BX$ by $j_k$. For a map $f\colon X\to Y$ where $Y$ is a topological monoid, let $\bar{f}\colon\Sigma X\to BY$ denote the adjoint of $f$. The following is shown in \cite{IM} (cf. \cite{HKT}).

\begin{lemma}
  \label{A_n}
  Let $X$ be an $A_n$-space and $Y$ be a topological monoid. A map $f\colon X\to Y$ is an $A_n$-map if and only if there is a map $P^nX\to BY$ satisfying a homotopy commutative diagram
  $$\xymatrix{\Sigma X\ar[r]^{\bar{f}}\ar[d]&BY\ar@{=}[d]\\
  P^nX\ar[r]&BY.}$$
\end{lemma}

Let $X_1,\ldots,X_l$ be $A_n$-spaces and let
$$\widehat{P}^n(X_1,\ldots,X_l)=\bigcup_{i_1+\cdots+i_l=n}P^{i_1}X_1\times\cdots\times P^{i_l}X_l.$$
We regard $X_1\times\cdots\times X_l$ as an $A_n$-space by the product of multiplications of $X_1,\ldots,X_l$. The following is proved in \cite{HKT}.

\begin{lemma}
  \label{retraction proj}
  Let $X_i$ be $A_n$-spaces for $i=1,2$. There are maps $\widehat{P}^i(X_1,X_2)\to P^i(X_1\times X_2)$ and $P^i(X_1\times X_2)\to\widehat{P}^i(X_1,X_2)$ for $i=1,2,\ldots,n$ satisfying a homotopy commutative diagram
  $$\xymatrix{\Sigma X_1\vee\Sigma X_2\ar[r]\ar[d]&\widehat{P}^2(X_1,X_2)\ar[r]\ar[d]&\cdots\ar[r]&\widehat{P}^n(X_1,X_2)\ar[d]\\
  \Sigma(X_1\times X_2)\ar[r]\ar[d]_{\Sigma p_1+\Sigma p_2}&P^2(X_1\times X_2)\ar[d]\ar[r]&\cdots\ar[r]&P^n(X_1\times X_2)\ar[d]\\
  \Sigma X_1\vee\Sigma X_2\ar[r]&\widehat{P}^2(X_1,X_2)\ar[r]&\cdots\ar[r]&\widehat{P}^n(X_1,X_2)}$$
  where the upper left arrow is the inclusion and $p_i\colon X_1\times X_2\to X_i$ is the $i^\text{th}$ projection for $i=1,2$.
\end{lemma}

In order to apply Lemma \ref{retraction proj} to our case, we need the following simple lemma. Let $X$ be an H-space. Then the projective space $P^2X$ is the cofiber of the Hopf construction $H\colon\Sigma X\wedge X\to\Sigma X$, where we write the inclusion $\Sigma X\to P^2X$ by $j$.

\begin{lemma}
  \label{retraction}
  Let $X_i$ be H-spaces for $i=1,2$. Then there is a homotopy commutative diagram
  $$\xymatrix{\Sigma(X_1\times X_2)\ar[r]^j\ar[d]_{\Sigma p_1+\Sigma p_2}&P^2(X_1\times X_2)\ar@{=}[d]\\
  \Sigma X_1\vee\Sigma X_2\ar[r]^{\hat{\jmath}}&P^2(X_1\times X_2)}$$
  where $p_i\colon X_1\times X_2\to X_i$ is the $i^\text{th}$ projection for $i=1,2$ and $\hat{\jmath}$ is the restriction of $j$.
\end{lemma}

\begin{proof}
  Let $h$ be the composite of the inclusion $\Sigma X_1\wedge X_2\to\Sigma(X_1\times X_2)\wedge(X_1\times X_2)$ and the Hopf construction $H\colon\Sigma(X_1\times X2)\wedge(X_1\times X_2)\to\Sigma(X_1\times X_2)$. By the definition of the Hopf construction, $h$ is the right homotopy inverse of the projection $\Sigma(X_1\times X_2)\to\Sigma(X_1\wedge X_2)$, and so the map
  $$\Sigma i_1\vee\Sigma i_2\vee h\colon\Sigma X_1\vee\Sigma X_2\vee\Sigma(X_1\wedge X_2)\to\Sigma(X_1\times X_2)$$
  is a homotopy equivalence, where $i_k\colon X_k\to X_1\times X_2$ is the inclusion for $k=1,2$. Let $r'$ be the composite of the homotopy inverse $(\Sigma i_1\vee\Sigma i_2\vee h)^{-1}$ and the projection $r\colon\Sigma X_1\vee\Sigma X_2\vee\Sigma(X_1\wedge X_2)\to\Sigma X_1\vee\Sigma X_2$. Then there is a homotopy cofibration
  $$\Sigma(X_1\wedge X_2)\xrightarrow{h}\Sigma(X_1\times X_2)\xrightarrow{r'}\Sigma X_1\vee\Sigma X_2$$
  and so one gets a homotopy commutative diagram
  $$\xymatrix{\Sigma(X_1\times X_2)\ar[r]^j\ar[d]_{r'}&P^2(X_1\times X_2)\ar@{=}[d]\\
  \Sigma X_1\vee\Sigma X_2\ar[r]^{\hat{\jmath}}&P^2(X_1\times X_2).}$$
  It remains to show $\hat{\jmath}\circ r'\simeq\hat{\jmath}\circ(\Sigma p_1+\Sigma p_2)$. Since $\Sigma X_1\times\Sigma X_2\subset \widehat{P}^2(X_1\times X_2)$, it follows from Lemma \ref{retraction proj} that $\hat{\jmath}$ factors through the inclusion $\Sigma i_1\vee\Sigma i_2\colon\Sigma X_1\vee\Sigma X_2\to\Sigma X_1\times \Sigma X_2$. Then it suffices to show $(\Sigma i_1\vee\Sigma i_2)\circ r'\simeq(\Sigma i_1\vee\Sigma i_2)\circ(\Sigma p_1+\Sigma p_2)$, or equivalently, $(\Sigma i_1\vee\Sigma i_2)\circ(\Sigma p_1+\Sigma p_2)\circ h\simeq*$. Now $(\Sigma i_1\vee\Sigma i_2)\circ(\Sigma p_1+\Sigma p_2)\circ h=\Sigma(i_1\circ p_1)\circ h+\Sigma(i_2\circ p_2)\circ h$. On the other hand, $\Sigma p_k\circ h\simeq*$ for $k=1,2$ by the definition of the Hopf construction. Thus $(\Sigma i_1\vee\Sigma i_2)\circ(\Sigma p_1+\Sigma p_2)\circ h\simeq*$ as desired.
\end{proof}

\begin{proposition}
  \label{A_n refined}
  Let $X_1,\ldots,X_l$ be $A_n$-spaces, $Y$ be a topological monoid, and $f\colon X_1\times\cdots\times X_l\to Y$be a map such that the restriction $f\vert_{X_i}$is an $A_n$-map for each $i$. Then $f$ is an $A_n$-map if and only if there is a map $\widehat{P}^n(X_1,\ldots,X_l)\to BY$ satisfying a homotopy commutative diagram
  $$\xymatrix{\Sigma X_1\vee\cdots\vee\Sigma X_l\ar[r]^(.7){\hat{f}}\ar[d]&BY\ar@{=}[d]\\
  \widehat{P}^n(X_1,\ldots,X_l)\ar[r]&BY}$$
  where $\hat{f}$ is the restriction of $\bar{f}$.
\end{proposition}

\begin{proof}
  We first consider the case $n=2$. The map $f$ is an $A_2$-map, that is, an H-map if and only if the Samelson products $\langle f\vert_{X_i},f\vert_{X_j}\rangle$ are trivial for all $i\ne j$. By the adjointness of Samelson products and Whitehead products, this is equivalent to that $\widehat{f}$ extends to a map
  $$\bigcup_{\substack{i_1+\cdots+i_l=2\\i_1,\ldots,i_l\le 1}}P^{i_1}X_1\times\cdots\times P^{i_l}X_l\to BY.$$
  Thus since each $f\vert_{X_i}$ is an H-map, such an extension exists if and only if $\hat{f}$ extends to $\widehat{P}^2(X_1,\ldots,X_l)\to BY$.

  We next consider the case $n\ge 3$. By the $n=2$ case, we may assume that $f$ is an H-map. Then by Lemma \ref{retraction}, there is a homotopy commutative diagram
  \begin{equation}
    \label{retraction P^2 B}
    \xymatrix{\Sigma(X_1\times\cdots\times X_l)\ar[d]_{\Sigma p_1+\cdots+\Sigma p_l}\ar[r]^j&P^2(X_1\times\cdots\times X_l)\ar@{=}[d]\ar[r]&P^2Y\ar@{=}[d]\ar[r]^{j_2}&BY\ar@{=}[d]\\
    \Sigma X_1\vee\cdots\vee\Sigma X_l\ar[r]^{\hat{\jmath}}&P^2(X_1\times\cdots\times X_l)\ar[r]&P^2Y\ar[r]^{j_2}&BY}
  \end{equation}
  where the composite of the upper row is $\bar{f}$ and the composite of the lower row is $\hat{f}$. Now suppose that there is a map $\widehat{P}^n(X_1,\ldots,X_l)\to BY$ extending $\hat{f}$. Then it follows from Lemmas \ref{retraction proj} and \ref{retraction} together with \eqref{retraction P^2 B} that there is a homotopy commutative diagram
  $$\xymatrix{\Sigma(X_1\times\cdots\times X_l)\ar[dd]_{\Sigma p_1+\cdots+\Sigma p_l}\ar[rr]^(.55){\bar{f}}\ar[rd]&&BY\ar@{=}[dd]\\
  &P^n(X_1\times\cdots\times X_n)\ar[dd]\\
  \Sigma X_1\vee\cdots\vee\Sigma X_l\ar[rr]^(.39){\hat{f}}|(.58)\hole\ar[rd]&&BY\\
  &\widehat{P}^n(X_1,\ldots,X_l)\ar[ur]}$$
  and thus by Lemma \ref{A_n} $f$ is an $A_n$-map.

  If $f$ is an $A_n$-map then by Lemma \ref{A_n} there is a map $P^n(X_1,\cdots,X_l)\to BY$ extending $\bar{f}$. Thus by Lemma \ref{retraction proj} one gets the desired map $\widehat{P}(X_1,\ldots,X_l)\to BY$. Therefore the proof is complete.
\end{proof}

Hereafter let $p$ be an odd prime and we localize at $p$. Let $G$ be a connected Lie group. By the classical result of Hopf, the rational cohomology of $G$ is an exterior algebra generated by odd degree elements. If generators are in dimensions $2n_1-1,\ldots,2n_r-1$ for $n_1\le\cdots\le n_r$ then we say that the type of $G$ is $(n_1,\ldots,n_r)$. Recall that $G$ is called $p$-regular if it is homotopy equivalent to a product of spheres such that
\begin{equation}
  \label{regular}
  G\simeq S^{2n_1-1}\times\cdots\times S^{2n_r-1}.
\end{equation}
It is well known that $G$ is $p$-regular if $p>n_r$. On the other hand, it is also well known that any odd sphere is an $A_{p-1}$-space. A cell decomposition of the projective spaces of odd sphres is given in \cite{HKT} as follows.

\begin{lemma}
  \label{P^lS}
  For $l\le p-1$ there is a $p$-local cell decomposition
  $$P^lS^{2n-1}\simeq S^{2n}\cup e^{4n}\cup\cdots\cup e^{2ln}.$$
\end{lemma}

The following proposition  is proved in \cite{HKT}, which shows an intrinsic feature of the decomposition \eqref{regular} with respect to higher homotopy associativity.

\begin{proposition}
  \label{regular A_k}
  Let $G$ be a connected Lie group of type $(n_1,\ldots,n_r)$. If $p>kn_r$, implying $G$ is $p$-regular, then the product $A_k$-structure on $S^{2n_1-1}\times\cdots\times S^{2n_r-1}$ and the standard $A_k$-structure on $G$ are equivalent.
\end{proposition}

We slightly improve this proposition in our setting. Let $(G,H)$ and $p$ be as in Theorem \ref{main general} and suppose $G$ is $p$-regular. Then $G/H$ is a product of odd spheres, and if
$$G/H\simeq S^{2l_1-1}\times\cdots\times S^{2l_t-1}$$
for $l_1\le\cdots\le l_t$ then we say that $G/H$ has type $(l_1,\ldots,l_t)$. Hereafter, let $(n_1,\ldots,n_r)$, $(m_1,\ldots,m_s)$ and $(l_1,\ldots,l_t)$  be the types of $G,\,H$ and $G/H$, respectively. Then $(m_1,\ldots,m_s)$ and $(l_1,\ldots,l_t)$ are subsequences of $(n_1,\ldots,n_r)$.

\begin{proposition}
  \label{regular A_k refined}
  Let $(G,H)$ and $p$ be as in Theorem \ref{main general}. If $p\ge b_k$ with $k\ge 2$ then the product $A_k$-structure on $H\times S^{2l_1-1}\times\cdots\times S^{2l_t-1}$ and the standard $A_k$-structure on $G$ are equivalent.
\end{proposition}

\begin{proof}
  Note that $G$ is $p$-regular for $p\ge b_k$. By Lemma \ref{A_n} and Proposition \ref{A_n refined} it suffices to show that the natural map $\Sigma(H\vee S^{2l_1-1}\vee\cdots\vee S^{2l_t-1})\to BG$ extends to a map $\widehat{P}^k(H,S^{2l_1-1},\ldots,S^{2l_t-1})\to BG$. Let
  $$X=\bigcup_{\substack{i+j=k\\i\ne k}}\widehat{P}^i(S^{2m_1-1},\ldots,S^{2m_s-1})\times\widehat{P}^j(S^{2l_1-1},\ldots,S^{2l_t-1}).$$
  We first construct an extension $X\to BG$ for $p\ge b_k$. Let $b_k=\max\{(k-1)m_s+l_t,kl_t\}$. By Lemma \ref{P^lS}, $X$ consists of even dimensional cells and $\dim X\le 2b_k$. Since $G$ is $p$-regular, its homotopy groups can be calculated from those of spheres in \cite{To}. In particular, $\pi_{2i-1}(BG)=0$ for $i\le p$, and so we get an extension $X\to BG$ for $p\ge b_k$.

  We next construct an extension $\widehat{P}^k(H,S^{2l_1-1},\ldots,S^{2l_t-1})\to BG$ from an extension $X\to BG$. Let
  $$Y=\widehat{P}^{k-1}(S^{2m_1-1},\ldots,S^{2m_s-1})\subset X.$$
  By construction, we may assume that the restriction of an extension $X\to BG$ to $Y$ decomposes as $Y\xrightarrow{f}BH\to BG$, where the second map is the induced map from the inclusion $H\to G$. By Proposition \ref{regular A_k} there is a map $g\colon P^{k-1}H\to Y$ such that the composite $f\circ g\colon P^{k-1}H\to BH$ restricts to the canonical map $\Sigma H\to BH$. Then $f\circ g$ gives the standard $A_{k-1}$-structure of the identity map of $H$, and so by \cite{Ts} there is an $A_{k-1}$-map $h\colon H\to H$ such that $f\circ g\circ P^{k-1}h$ is homotopic to the canonical map $j_{k-1}\colon P^{k-1}H\to BH$, where $P^{k-1}h\colon P^{k-1}H\to P^{k-1}H$ is the induced map from $h$. Then the composite
  $$\bigcup_{\substack{i+j=k\\i\ne k}}P^iH\times\widehat{P}^j(S^{2l_1-1},\ldots,S^{2l_t-1})\xrightarrow{g\circ P^{k-1}h\times 1}X\to BG$$
  restricts to the canonical map $j_{k-1}\colon P^{k-1}H\to BG$. Since the canonical map $j_{k-1}$ extends to $j_k\colon P^kH\to BG$, we finally obtain the desired extension $\widehat{P}^k(H,S^{2l_1-1},\ldots,S^{2l_t-1})\to BG$, completing the proof.
\end{proof}

Now we prove Theorem \ref{main general}.

\begin{proof}
  [Proof of Theorem \ref{main}]
  By Lemma \ref{A_n} and Proposition \ref{regular A_k refined} it suffices to show that the adjoint $\bar{q}\colon\Sigma G\to BH$ of the projection $q\colon G\to H$ extends to $\widehat{P}^k(H,S^{2l_1-1},\ldots,S^{2l_t-1})\to BH$. Note that $q\colon G\to H$ is identified with the projection
  $$H\times S^{2l_1-1}\times\cdots\times S^{2l_t-1}\to H.$$
  Then there is a homotopy commutative diagram
  $$\xymatrix{\Sigma G\ar[r]^{\Sigma q}\ar[d]&\Sigma H\ar[d]\\
  \widehat{P}^k(H,S^{2l_1-1},\ldots,S^{2l_t-1})\ar[r]&P^kH}$$
  where the bottom map is the projection. Thus since the composite $\Sigma G\xrightarrow{\Sigma q}\Sigma H\to BH$ is $\bar{q}$, the composite
  $$\widehat{P}^k(H,S^{2l_1-1},\ldots,S^{2l_t-1})\to P^kH\xrightarrow{j_k}BH$$
  is the desired extension of $\bar{q}$.
\end{proof}


\section{Cohomology calculation}

For the rest of the paper, cohomology is assumed to be with mod $p$ coefficients for an odd prime $p$. This section calculates $\mathcal{P}^1x$ for some cohomology classes $x$ of the classifying spaces of Lie groups, which we are going to use. We first consider $SU(n)$. The cohomology of $BSU(n)$ is given by
\begin{equation}
  \label{H(BSU)}
  H^*(BSU(n))=\Z/p[c_2,\ldots,c_n]
\end{equation}
where $c_i$ is the Chern class. In \cite{Sh}, $\mathcal{P}^ic_k$ is determined, and in particular,
\begin{multline}
  \label{mod p Wu}
  \mathcal{P}^1c_k=\sum_{2i_2+3i_3+\cdots+ni_n=k+p-1}(-1)^{i_2+\cdots+i_n-1}\frac{(i_2+\cdots+i_n-1)!}{i_2!\cdots i_n!}\\
  \times\left(k-1-\frac{\sum_{j=2}^{k-1}(k+p-1-j)i_j}{i_2+\cdots+i_n-1}\right)c_2^{i_2}\cdots c_n^{i_n}.
\end{multline}
If a polynomial $P$ includes a monomial $M$ then we write $P\ge M$. By \eqref{mod p Wu} one gets:

\begin{lemma}
  \label{P^1 SU}
  Let $k\ge 2$.
  \begin{enumerate}
    \item In $H^*(BSU(2n+1))$,
    \begin{align*}
      \mathcal{P}^1c_2&\ge
      \begin{cases}
        (-1)^{k-1}c_{2n+1}^{k-1}c_{p-(k-1)(2n+1)+1}&((k-1)(2n+1)\le p< k(2n+1)-1)\\
        (-1)^{k-1}\frac{1}{k}c_{2n+1}^k&(p=k(2n+1)-1)\\
      \end{cases}\\
      \mathcal{P}^1c_4&\ge -3c_{2n+1}c_3\quad(p=2n+1).
    \end{align*}
    \item In $H^*(BSU(2n))$ with $k,n\ge 2$,
    \begin{align*}
      \mathcal{P}^1c_2&\ge(-1)^{k-1}(k-1)c_{2n}^{k-2}c_{2n-1}c_{p-2(k-1)n+2}&&(2(k-1)n\le p<2kn-1)\\
      \mathcal{P}^1c_4&\ge(-1)^{k-1}3(k-1)c_{2n}^{k-2}c_{2n-1}c_3&&(p=2(k-1)n-1).
    \end{align*}
  \end{enumerate}
\end{lemma}

We next consider $SO(2n+1)$. The cohomology of $BSO(2n+1)$ is given by
\begin{equation}
  \label{H(BSO)}
  H^*(BSO(2n+1))=\Z/p[p_1,\ldots,p_n]
\end{equation}
where $p_i$ is the Pontrjagin class. The inclusion $c\colon SO(2n+1)\to SU(2n+1)$ satisfies $c^*(c_{2k})=(-1)^kp_k$ and $c^*(c_{2k+1})=0$, and so by \eqref{mod p Wu},
\begin{multline}
  \label{mod p Wu Pontrjagin}
  \mathcal{P}^1p_k=\sum_{i_1+2i_2+\cdots+ni_n=k+\frac{p-1}{2}}(-1)^{i_1+\cdots+i_n+\frac{p+1}{2}}\frac{(i_1+\cdots+i_n-1)!}{i_1!\cdots i_n!}\\
  \times\left(2k-1-\frac{\sum_{j=1}^{k-1}(2k+p-1-2j)i_j}{i_1+\cdots+i_n-1}\right)p_1^{i_1}\cdots p_n^{i_n}.
\end{multline}
Thus by focusing on $BSO(15)$, one gets the following.

\begin{lemma}
  \label{P^1 SO(15)}
  Let $k\ge 2$. In $H^*(BSO(15))$ the following hold:
  \begin{align*}
  \mathcal{P}^1p_1&\ge
  \begin{cases}
    (-1)^{k+\frac{p-1}{2}}\frac{5}{12}p_7p_6^{k-2}p_2&(p=12k-7)\\
    (-1)^{k+\frac{p-1}{2}}p_7p_6^{k-2}&(p=12k-11)
  \end{cases}\\
  \mathcal{P}^1p_2&\ge(-1)^{k+\frac{p-1}{2}}3p_7p_6^{k-2}\quad(p=12k-13)\\
  \mathcal{P}^1p_4&\ge
  \begin{cases}
    (-1)^{k+\frac{p-1}{2}}\frac{29}{12}p_7p_6^{k-2}p_2&(p=12k-13)\\
    (-1)^{k+\frac{p-1}{2}}7p_7p_6^{k-2}&(p=12k-17).
  \end{cases}
  \end{align*}
\end{lemma}

The cohomology of $BSO(2n)$ is given by
$$H^*(BSO(2n))=\Z/p[p_1,\ldots,p_{n-1},e_n]$$
where $p_i$ is the Pontrjagin class and $e_n$ is the Euler class. Then the inclusion $j\colon SO(2n)\to SO(2n+1)$ satisfies $j^*(p_i)=p_i$ for $i=1,\ldots,n-1$ and $j^*(p_n)=e_n^2$. Thus by \eqref{mod p Wu Pontrjagin} one gets:

\begin{lemma}
  \label{P^1 SO(2n)}
  In $H^*(BSO(2n))$, for $2(k-2)(n-1)+2\le p<2(k-1)(n-1)+2$ with $k\ge 3$
    $$\mathcal{P}^1p_1\ge (-1)^{k+\frac{p-1}{2}}(k-2)p_{n-1}^{k-3}p_{\frac{p-1}{2}-(k-2)(n-1)}e_n^2$$
    and for $p\le 2n-1$
    $$\mathcal{P}^1p_{n-\frac{p-1}{2}}\ge (-1)^\frac{p-1}{2}2ne_n^2.$$
\end{lemma}

\begin{lemma}
  \label{SO(8)}
  Let $k\ge 2$. In $H^*(BSO(8))$ the following hold:
  \begin{align*}
    \mathcal{P}^1p_1&\ge(-1)^{k+\frac{p-1}{2}}\frac{1}{12}p_3^{k-2}p_2^2&&(p=6k-5)\\
    \mathcal{P}^1p_2&\ge(-1)^{k+\frac{p+1}{2}}\frac{1}{4}p_3^{k-2}p_2^2&&(p=6k-7).
  \end{align*}
\end{lemma}

We finally consider the exceptional Lie groups. Let $p$ be a prime $>5$. The cohomology of $BE_8$ is given by
$$H^*(BE_8)=\Z/p[x_4,x_{16},x_{24},x_{28},x_{36},x_{40},x_{48},x_{60}],\quad|x_i|=i.$$
Let $j_2\colon Spin(15)\to E_8$ be the canonical inclusion. It is shown in \cite{HKO} that the generators $x_i$ can be chosen such that
\begin{align*}
  j_2^*(x_4)&=p_1\\
  j_2^*(x_{16})&=12p_4-\frac{18}{5}p_3p_1+p_2^2+\frac{1}{10}p_2p_1^2\\
  j_2^*(x_{24})&=60p_6-5p_5p_1-5p_4p_2+3p_3^2-p_3p_2p_1+\frac{5}{36}p_2^3\\
  j_2^*(x_{28})&=480p_7+40p_5p_2-12p_4p_3-p_3p_2^2-3p_4p_2p_1+\frac{24}{5}p_3^2p_1+\frac{11}{36}p_2^3p_1\\
  j_2^*(x_{36})&\equiv 480p_7p_2+72p_6p_3-30p_5p_4-\frac{25}{2}p_5p_2^2+9p_4p_3p_2-\frac{18}{5}p_3^3-\frac{1}{4}p_3p_2^3\mod(p_1)
\end{align*}
where $BSpin(15)\simeq BSO(15)$ since we are localizing at an odd prime. Then in particular, $x_{60}$ can be chosen such that $j_2^*(x_{60})$ does not include the monomial $ap_7p_6p_2$ for $a\in\Z/p$. Thus by Lemma \ref{P^1 SO(15)} and a degree reason, one gets:

\begin{lemma}
  \label{P^1 E_8}
  Let $p$ be a prime $>5$ and $k\ge 2$. In $H^*(BE_8)$,
  \begin{align*}
    \mathcal{P}^1x_4&\ge
    \begin{cases}
      (-1)^{k+\frac{p-1}{2}}\frac{5}{96\cdot 60^{k-1}}x_{36}x_{24}^{k-2}&(p=12k-7)\\
      (-1)^{k+\frac{p+1}{2}}\frac{1}{8\cdot 60^{k-1}}x_{28}x_{24}^{k-2}&(p=12k-11)
    \end{cases}\\
    \mathcal{P}^1x_{16}&\ge
    \begin{cases}
      (-1)^{k+\frac{p-1}{2}}\frac{35}{8\cdot 60^{k-1}}x_{36}x_{24}^{k-2}&(p=12k-13)\\
      (-1)^{k+\frac{p-1}{2}}\frac{7}{8\cdot 60^{k-1}}x_{28}x_{24}^{k-2}&(p=12k-17).
    \end{cases}
  \end{align*}
\end{lemma}

There is a commutative square of inclusions
\begin{equation}
  \label{E_6 E_8}
  \xymatrix{Spin(10)\ar[r]^{i_1}\ar[d]_{j_1}&Spin(15)\ar[d]^{j_2}\\
  E_6\ar[r]^{i_2}&E_8.}
\end{equation}
Let $p$ be a prime $>5$. The cohomology of $BE_6$ and $BSO(2m)$ are given by
\begin{equation}
  \label{H(BE_6)}
  H^*(BE_6)=\Z/p[x_4,x_{10},x_{12},x_{16},x_{18},x_{24}],\quad|x_i|=i.
\end{equation}
It is shown in \cite{HKO} that $x_{10}$ and $x_{18}$ are chosen as
$$j_1^*(x_{10})=e_5,\quad j_1^*(x_{12})=-6p_3+p_2p_1,\quad j_1^*(x_{18})=p_2e_5$$
and that $i_2^*(x_i)=x_i$ for $i=4,12,16,24$.
Then by the choice of the generators $x_i$ of $H^*(BE_8)$ and the commutative diagram \eqref{E_6 E_8}, $$i_2^*(x_{28})=40x_{18}x_{10}+\frac{1}{6}x_{16}x_{12},\quad i_2^*(x_{36})=-10x_{18}^2-\frac{5}{2}x_{16}x_{10}^2.$$
Thus by Lemma \ref{P^1 E_8} one finally obtains the following.

\begin{corollary}
  \label{P^1 E_6}
  Let $p$ be a prime $>5$ and $k\ge 2$. In $H^*(BE_6)$,
  \begin{align*}
    \mathcal{P}^1x_4&\ge
    \begin{cases}
      (-1)^{k+\frac{p+1}{2}}\frac{25}{48\cdot 60^{k-1}}x_{24}^{k-2}x_{18}^2&(p=12k-7)\\
      (-1)^{k+\frac{p+1}{2}}\frac{5}{60^{k-1}}x_{24}^{k-2}x_{18}x_{10}&(p=12k-11)
    \end{cases}\\
    \mathcal{P}^1x_{16}&\ge
    \begin{cases}
      (-1)^{k+\frac{p+1}{2}}\frac{175}{4\cdot 60^{k-1}}x_{24}^{k-2}x_{18}^2&(p=12k-13)\\
      (-1)^{k+\frac{p-1}{2}}\frac{35}{60^{k-1}}x_{24}^{k-2}x_{18}x_{10}&(p=12k-17).
    \end{cases}
  \end{align*}
\end{corollary}


\section{Proof of the main theorem}

Let $\langle a,b\rangle$ denote the Samelson product of $a,b\in\pi_*(X)$ for an H-group $X$. We will need the following lemma.

\begin{lemma}
  \label{Samelson product}
  Let $f\colon X\to Y$ be a map between H-groups. Suppose there are $a,b\in\pi_*(X)$ such that $f_*(a)=0$ and $f_*(\langle a,b\rangle)\ne 0$. Then $f$ is not an H-map.
\end{lemma}

\begin{proof}
  If $f$ is an H-map then by the naturality of Samelson products
  $$0\ne f_*(\langle a,b\rangle)=\langle f_*(a),f_*(b)\rangle=0,$$
  which is a contradiction. Thus $f$ is not an H-map.
\end{proof}

All Samelson products that we need to apply Lemma \ref{Samelson product} are calculated in \cite{B,HKMO} except for the case $(G,H)=(Spin(8),G_2)$ and $p=2$. Then we calculate a certain Samelson product in $Spin(8)$ for $p=2$. Recall that the fibration $Spin(7)\to Spin(8)\to S^7$ is trivial such that $Spin(8)=Spin(7)\times S^7$. Let $\partial\colon\pi_{*+1}(S^8)\to\pi_*(Spin(8))$ be the connecting map of the homotopy exact sequence of a fibration $Spin(8)\to Spin(9)\to S^8$. We wll freely use the notation of the homotopy groups of spheres in \cite{To}.

\begin{lemma}
  \label{Samelson 1}
  Let $p=2$ and $[\iota_7]\colon S^7\to Spin(7)\times S^7=Spin(8)$ be the inclusion. Then
  $$\langle[\iota_7],[\iota_7]\rangle=\pm\partial(2\sigma_8-\Sigma\sigma').$$
\end{lemma}

\begin{proof}
  Consider a commutative diagram with fibration rows and columns
  $$\xymatrix{Spin(7)\ar[d]\ar@{=}[r]&Spin(7)\ar[d]\ar[r]&\ast\ar[d]\\
  Spin(8)\ar[r]\ar[d]&Spin(9)\ar[r]\ar[d]&S^8\ar@{=}[d]\\
  S^7\ar[r]&S^{15}\ar[r]&S^8.}$$
  Then in the homotopy exact sequence of the middle row
  $$\cdots\to\pi_{*+1}(S^8)\xrightarrow{\partial}\pi_*(Spin(8))\to\pi_*(Spin(9))\to\pi_*(S^8)\to\cdots$$
  one has $\partial(\iota_8)=[\iota_7]$, where $\iota_8$ is the identity map of $S^8$. Thus by the adjointness of Whitehead products and Samelson products, $\partial([\iota_8,\iota_8])=\langle[\iota_7],[\iota_7]\rangle$. On the other hand, it is shown in \cite[p. 50]{To} that $\pi_{15}(S^8)=\Z\{\sigma_8\}\oplus\Z/8\{\Sigma\sigma'\}$ and $[\iota_8,\iota_8]=\pm(2\sigma_8-\Sigma\sigma')$, where for a cyclic group $A$, $A\{x\}$ is a cyclic group generated by $x$ which is isomorphic to $A$. Thus one gets the desired equality.
\end{proof}

There is a commutative diagram of inclusions
$$\xymatrix{SU(3)\ar[rr]^i\ar@{=}[d]&&G_2\ar[d]^j\\
SU(3)\ar[r]^{i'}&Spin(6)\ar[r]^{j'}&Spin(7)\ar[r]^{j''}&Spin(8).}$$
It is shown in \cite{KM} that
$$\pi_{14}(Spin(7))=\Z/8\{[\bar{\nu}_6+\epsilon_6]\}\oplus\Z/2\{[\nu_5]_7\nu_8^2\}\oplus\Z/8\{[\eta_5\epsilon_6]_7\}.$$
Then since $Spin(8)=Spin(7)\times S^7$,
$$\pi_{14}(Spin(8))=\Z/8\{j''_*([\bar{\nu}_6+\epsilon_6])\}\oplus\Z/2\{j''_*([\nu_5]_7\nu_8^2)\}\oplus\Z/8\{j''_*([\eta_5\epsilon_6]_7)\}\oplus\Z/8\{[\iota_7]\sigma'\},$$
where $\pi_{14}(S^7)=\Z/8\{\sigma'\}$. The following is proved in \cite{KM}.

\begin{lemma}
  \label{eq 1}
  For an odd integer $u_1$ and an integer $u_2$,
  \begin{align*}
    \partial(\sigma_8)&\equiv [\iota_7]\sigma'+u_1j''_*([\bar{\nu}_6+\epsilon_6])+u_2j''_*([\eta_5\epsilon_6]_7)&&\mod j''_*([\nu_5]_7\nu_8^2)\\
    \partial(\Sigma\sigma')&\equiv 2[\iota_7]\sigma'+4j''_*([\bar{\nu}_6+\epsilon_6])-j''_*([\eta_5\epsilon_6]_7)&&\mod j''_*([\nu_5]_7\nu_8^2).
  \end{align*}
\end{lemma}

To proceed the calculation, we relate the above generators of $\pi_{14}(Spin(8))$ with $G_2$ in $Spin(8)$. As in \cite{Mi1},
$$\pi_{14}(G_2)=\Z/8\{\langle\bar{\nu}_6+\epsilon_6\rangle\}\oplus\Z/2\{i_*([\nu_5^2]\nu_{11})\}.$$
Since $\pi_6(G_2)=0$ as in \cite{Mi1}, the fibration $G_2\xrightarrow{i}Spin(7)\to S^7$ splits such that $Spin(7)\simeq G_2\times S^7$. Then
$$\pi_{14}(Spin(7))=\Z/8\{j_*(\langle\bar{\nu}_6+\epsilon_6\rangle)\}\oplus\Z/2\{j_*(i_*([\nu_5^2]\nu_{11}))\}\oplus\Z/8\{[\iota_7']\sigma'\}.$$
where $[\iota_7']\colon S^7\to G_2\times S^7\simeq Spin(7)$ is the inclusion. In particular, for some integers $s_1,s_2,s_3$,
\begin{equation}
  \label{eq 2}
  [\eta_5\epsilon_6]_7=s_1j_*(\langle\bar{\nu}_6+\epsilon_6\rangle)+s_2j_*(i_*([\nu_5^2]\nu_{11}))+s_3[\iota_7']\sigma'.
\end{equation}

\begin{lemma}
  \label{eq 3}
  We may choose $[\nu_5]_7\in\pi_8(Spin(7))$ and $[\nu_5^2]\in\pi_{11}(SU(3))$ such that
  $$[\nu_5]_7\nu_8^2=j_*(i_*([\nu_5^2]\nu_{11})).$$
\end{lemma}

\begin{proof}
  Consider a commutative diagram with fibration rows
  $$\xymatrix{SU(2)\ar[r]\ar[d]&SU(3)\ar[r]^{\bar{\pi}}\ar[d]^{i'}&S^5\ar@{=}[d]\\
  Spin(5)\ar[r]&Spin(6)\ar[r]^{\pi'}&S^5.}$$
  As in \cite{KM}, there is $x\in\pi_8(Spin(6))$ such that $\pi'_*(x)=\nu_5$, and $[\nu_5]_7\in\pi_8(Spin(7))$ is chosen as $[\nu_5]_7=j'_*(x)$, where $\pi_{n+3}(S^n)=\Z/8\{\nu_n\}$ for $n\ge 5$. On the other hand, in \cite{Mi1}, $[\nu_5^2]\in\pi_{11}(SU(3))$ is chosen to be any $y\in\pi_{11}(SU(3))$ satisfying $\bar{\pi}_*(y)=\nu_5^2$. Since $Spin(6)/SU(3)=S^7$ and $\pi_k(S^7)=0$ for $k=11,12$, $i'$ is an isomorphism in $\pi_{11}$. Then since $\pi'_*(x\nu_8)=\nu_5^2$, we may put $[\nu_5^2]=(i'_*)^{-1}(x\nu_8)$. Thus
  $$j_*(i_*([\nu_5^2]\nu_{11}))=j'_*(x\nu_8^2)=[\nu_5]_7\nu_8^2$$
  as desired.
\end{proof}

\begin{lemma}
  \label{Samelson product Spin(8)}
  Let $q\colon Spin(8)\to G_2$ be the projection
  $$Spin(8)=Spin(7)\times S^7\simeq G_2\times S^7\times S^7\to G_2.$$
  Then $q_*(\langle[\iota_7],[\iota_7]\rangle)\ne 0$.
\end{lemma}

\begin{proof}
  Consider a commutative diagram with fibration rows
  \begin{equation}
    \label{G_2 Spin(7)}
    \xymatrix{SU(3)\ar[r]^i\ar[d]&G_2\ar[r]^\pi\ar[d]^{j}&S^6\ar@{=}[d]\\
    Spin(6)\ar[r]^{j'}&Spin(7)\ar[r]^{\pi''}&S^6.}
  \end{equation}
  In \cite{Mi1,KM}, $\langle\bar{\nu}_6+\epsilon_6\rangle$ and $[\bar{\nu}_6+\epsilon_6]$ are chosen such that $\pi_*(\langle\bar{\nu}_6+\epsilon_6\rangle)=\pi''_*([\bar{\nu}_6+\epsilon_6])=\bar{\nu}_6+\epsilon_6$, where $\pi_{14}(S^6)=\Z/8\{\bar{\nu}_6\}\oplus\Z/2\{\epsilon_6\}$. Then by Lemma \ref{eq 3},
  $$j_*(\langle\bar{\nu}_6+\epsilon_6\rangle)=[\bar{\nu}_6+\epsilon_6]+t_1j_*(i_*([\nu_5^2]\nu_{11}))+t_2[\eta_5\epsilon_6]_7$$
  for some integers $t_1,t_2$. Thus by \eqref{eq 2} and Lemmas \ref{Samelson 1}, \ref{eq 1}, and \ref{eq 3},
  \begin{multline*}
    \langle[\iota_7],[\iota_7]\rangle\equiv\pm(2u_1-4+s_1-2(u_1t_2-t_2-u_2)s_1)j''_*(j_*(\langle\bar{\nu}_6+\epsilon_6\rangle))\\
    \mod(j''_*(j_*(i_*([\nu_5^2]\nu_{11}))),j''_*([\iota_7'])\sigma').
  \end{multline*}
  Since $q\circ j''\circ j\simeq 1$ and $q_*(j''_*([\iota_7']\sigma'))=0$,
  $$q_*(\langle[\iota_7],[\iota_7])\equiv\pm(2u_1-4+s_1-2(u_1t_2-t_2-u_2)s_1)\langle\bar{\nu}_6+\epsilon_6\rangle\mod i_*([\nu_5^2]\nu_{11}),$$
  and so it suffices to show  $s_1\equiv 4\mod 8$.

  By the definition of $[\eta_5\epsilon_6]_7$ in \cite{KM}, $\pi''_*([\eta_5\epsilon_6]_7)=0$, and so $s_1(\bar{\nu}_6+\epsilon_6)+s_3\pi''_*([\iota_7']\sigma')=0$. Consider a homotopy exact sequence of the lower fibration of \eqref{G_2 Spin(7)}. Since $Spin(6)=SU(4)$ and $\pi_6(SU(6))\cong\pi_6(SU(\infty))=0$, one gets $\pi_6(Spin(6))=\pi_6(SU(4))=0$, and so $\pi''_*\colon\pi_7(Spin(7))\to\pi_7(S^6)$ is surjective. By \cite{To} and \cite{Mi1}, $\pi_7(S^6)=\Z/2\{\eta_6\}$ and $\pi_7(Spin(7))=\Z\{[\iota'_7]\}$, implying $\pi''_*([\iota'_7])=\eta_6$. Thus $s_1(\bar{\nu}_6+\epsilon_6)+s_3\eta_6\sigma'=0$. Now by \cite[p. 64]{To}, $\eta_6\sigma'=4\bar{\nu}_6$, and so $(s_1+4s_3)\bar{\nu}_6+s_1\epsilon_6=0$. Then $s_1$ is even and $s_1+4s_3\equiv 0\mod 8$. On the other hand, since $[\eta_5\epsilon_6]_7$ has order 8, either $s_1$ or $s_3$ must be odd. Then $s_3$ is odd, implying $s_1\equiv 4\mod 8$ as desired.
\end{proof}

\begin{proposition}
  \label{a_1}
  Let $(G,H),\,a_k$ and $p$ be as in Theorem \ref{main}. If $p<a_1$ then the projection $q\colon G\to H$ is not an H-map.
\end{proposition}

\begin{proof}
  Let $(G,H)=(SU(2n+1),SO(2n+1))$. Then $p<2n+1$. Since $SO(2n+1)$ is a direct summand of $SU(2n+1)$, $\pi_*(SO(2n+1))$ is a direct summand of $\pi_*(SU(2n+1))$ as well. As in \cite{Mi2} $\pi_{4n+2}(SU(2n+1))$ is a cyclic group. Since $SO(2n+1)\simeq Sp(n)$ where we are localizing at an odd prime $p$, $\pi_{4n+2}(SO(2n+1))\ne 0$ as in \cite{Mi2}. Then the projection $q$ induces an isomorphism in $\pi_{4n+2}$. Let $\epsilon_i$ be a generator of $\pi_{2i+1}(SU(2n+1))\cong\Z_{(p)}$ for $i=1,\ldots,2n-1$. By \cite{B} that the Samelson product $\langle\epsilon_{2n-p+1},\epsilon_{p-1}\rangle$ in $SU(2n+1)$ is non-trivial. Then in particular, since $q$ is an isomorphism in $\pi_{4n+2}$, $q_*(\langle\epsilon_{2n-p+1},\epsilon_{p-1}\rangle)\ne 0$. On the other hand, since $\pi_{2p-1}(SO(2n+1))=0$ for $p<2n+1$, $q_*(\epsilon_{p-1})=0$. Thus the proof is done by Lemma \ref{Samelson product}.

  The case $(G,H)=(SU(2n),Sp(n))$ follows from the same argument using $\langle\epsilon_{2n-p+1},\epsilon_{p-1}\rangle$ in $SU(2n)$, where $p<2n-1$. There is nothing to do for $(G,H)=(SO(2n),SO(2n-1))$ since $a_1=1$.

  Let $(G,H)=(E_6,F_4)$. Then $p=5$, and so
  $$E_6\simeq F_4\times B(9,17)$$
  as in \cite{Mi2} such that the projection $q\colon E_6\to F_4$ is identified with the projection $F_4\times B(9,17)\to F_4$, where $B(9,17)$ is a certain $S^9$-bundle over $S^{17}$. Let $\epsilon\colon B(9,17)\to E_6$ be the inclusion. In \cite{HKMO} it is shown that the Samelson product $\langle\epsilon,\epsilon\rangle$ is non-trivial, and its proof actually shows that $\langle\epsilon\vert_{S^9},\epsilon\vert_{S^9}\rangle$ is non-trivial, where $\epsilon\vert_{S^9}$ is the restriction of $\epsilon$ to the bottom cell $S^9\subset B(9,17)$. The homotopy groups of $B(9,17)$ are calculated in \cite{K} such that $\pi_{18}(B(9,17))=0$. Then $q_*(\langle\epsilon\vert_{S^9},\epsilon\vert_{S^9}\rangle)\ne 0$. On the other hand, $q_*(\epsilon\vert_{S^9})=0$, and so by Lemma \ref{Samelson product} $q$ is not an H-map.

  Let $(G,H)=(Spin(8),G_2)$. Then $p=2$. By \cite{Mi1}, the inclusion $G_2\to Spin(8)$ is injective in $\pi_{14}$. Then by Lemmas \ref{Samelson product} and \ref{Samelson product Spin(8)}, $q$ is not an H-map. Thus the proof is complete.
\end{proof}

We set notation on cohomology. Let $G$ be a connected Lie group whose integral homology has no $p$-torsion. Recall from \cite{He} that there is an isomorphism of unstable algebras
\begin{equation}
  \label{PG}
  H^*(P^kG)\cong H^*(BG)/\widetilde{H}^*(BG)^{k+1}\oplus S
\end{equation}
for some unstable algebra $S$ depending on $G$ and $k$, where the natural map $j_k\colon P^kG\to BG$ induces the obvious projection in cohomology. Let $(G,H)$ and $p$ be in Theorem \ref{main}. The cohomology of $BG$ and $BH$ are given by
\begin{equation}
  \label{H(BG)}
  H^*(BG)=\Z/p[x_{n_1},\ldots,x_{n_r}]\quad\text{and}\quad H^*(BH)=\Z/p[y_{m_1},\ldots, y_{m_s}]
\end{equation}
where $|x_i|=|y_i|=2i$.

If $m_i$ in $(m_1,\ldots,m_s)$ and $n_j$ in $(n_1,\ldots,n_r)$ are equal as a member of the sequence $(n_1,\ldots,n_r)$, then we say that $m_i=n_j$ in the type of $G$, where $(m_1,\ldots,m_s)$ is a subsequence of $(n_1,\ldots,n_r)$. Now we state a cohomological criterion for the projection $G\to H$ not being an $A_l$-map.

\begin{lemma}
  \label{P^1 criterion}
  Let $(G,H)$ and $p$ be as in Theorem \ref{main}, and let $j\colon H\to G$ be the inclusion. Suppose that presentations \eqref{H(BG)} satisfy
  $$j^*(x_{n_i})=
  \begin{cases}
    \pm y_{m_j}&n_i=m_j\text{ in the type of }G\\
    0&\text{otherwise}.
  \end{cases}$$
  Suppose also that there are an integer $m_k$ in the type of $H$ and a monomial $M=ax_{n_{i_1}}\cdots x_{n_{i_l}}\in H^*(BG)$ with $a\ne 0$ and $n_{i_1}\le\cdots\le n_{i_l}$ satisfying the following conditions:
  \begin{enumerate}
    \item $n_{i_j}>m_k$ for $j\ge 2$ and $m_k-n_{i_1}$ is not a sum of integers in the type of $G$;
    \item for any map $f\colon H^*(BH)\to H^*(BG)$ satisfying $f(y_{m_i})\equiv\pm x_{m_i}\mod\widetilde{H}^*(BG)^2$ for all $i$, there is no polynomial $P\in\mathrm{Im}\,f$ such that
    $$P\ge M\mod I$$
    where $m_k=n_h$ in the type of $G$ and $I=(x_{n_1},\ldots,\widehat{x_{n_{i_1}}},\ldots,x_{n_h})$.
  \end{enumerate}
  Then the projection $q\colon G\to H$ is not an $A_l$-map.
\end{lemma}

\begin{proof}
  We assume $q$ is an $A_l$-map and show a contradiction. By Lemma \ref{A_n} the adjoint $\bar{q}\colon\Sigma G\to BH$ extends to $\bar{q}_l\colon P^lG\to BH$. Let $S$ be as in \eqref{PG}. By the condition (1),
  \begin{equation}
    \label{q_l}
    \bar{q}_l^*(y_{m_i})\equiv\pm x_{m_i}\mod I^2+S.
  \end{equation}
  Then by the Cartan formula,
  \begin{equation}
    \label{M}
    \bar{q}_l^*(\mathcal{P}^1y_{m_k})\equiv\mathcal{P}^1\bar{q}_l^*(y_{m_k})\equiv\mathcal{P}^1x_{m_k}\ge M\not\equiv 0\mod I+S.
  \end{equation}
  Let $f\colon H^*(BH)\to H^*(BG)$ be the composite of a lift of $\bar{q}_l^*\colon H^*(BH)\to H^*(P^lG)$ through the projection $H^*(BG)\oplus S\to H^*(P^lG)$ under the isomorphism \eqref{PG} and the projection $H^*(BG)\oplus S\to H^*(BG)$. Then by \eqref{q_l}, $f(y_{m_i})\equiv\pm x_{m_i}\mod\widetilde{H}^*(BG)^2$, and so by condition (2) there is no $P\in\mathrm{Im}\,f$ such that $P\ge M\mod I$. Thus there is no $Q\in\mathrm{Im}\,\bar{q}_l^*$ such that $Q\ge M\mod I+S$, which contradicts \eqref{M}. Therefore the proof is complete.
\end{proof}

\begin{proposition}
  \label{a_k}
  Let $(G,H),\,a_k$ and $p$ be as in Theorem \ref{main}, and let $k\ge 2$. Then the following statements hold:
  \begin{enumerate}
    \item if $a_{k-1}\le p<a_k$ then the projection $G\to H$ is not an $A_k$-map except for $(G,H)=(SO(2n),SO(2n-1))$;
    \item if $a_{k-1}-n+2\le p<a_k-n+2$ then the projection $q\colon SO(2n)\to SO(2n-1)$ is not an $A_k$-map.
  \end{enumerate}
\end{proposition}

\begin{proof}
  (1) Let $(G,H)=(SU(2n+1),SO(2n+1))$. The inclusion $c\colon SO(2n+1)\to SU(2n+1)$ satisfies $c^*(c_{2i})=(-1)^ip_i$ and $c^*(c_{2i+1})=0$. Then the presentations \eqref{H(BSU)} and \eqref{H(BSO)} satisfy the condition in Lemma \ref{P^1 criterion}. By a dimensional consideration, we see that the monomial $(-1)^{k-1}c_{2n+1}^{k-1}c_{p-(k-1)(2n+1)+1}$ satisfies the condition (2) of Lemma \ref{P^1 criterion}. Thus by Lemmas \ref{P^1 SU} and \ref{P^1 criterion}, the projection $q$ is not an $A_k$-map for $a_{k-1}\le p<a_k$.

  Let $(G,H)=(SU(2n),Sp(n))$. The cohomology of $BSp(n)$ is given by
  $$H^*(BSp(n))=\Z/p[q_1,\ldots,q_n]$$
  where $q_i$ is the symplectic Pontrjagin class. Then the inclusion $c'\colon Sp(n)\to SU(2n)$ satisfies $(c')^*(c_{2i})=(-1)^iq_i$ and $(c')^*(c_{2i+1})=0$. Thus by Lemmas Lemmas \ref{P^1 SU} and \ref{P^1 criterion}, $q$ is not an $A_k$-map for $a_{k-1}\le p<a_k$.

  Let $(G,H)=(E_6,F_4)$. The cohomology of $BF_4$ is given by
  $$H^*(BF_4)=\Z/p[x_4,x_{12},x_{16},x_{24}],\quad|x_i|=i.$$
  In \cite{HKO}, it is shown that we may choose generators of $H^*(BF_4)$ such that the inclusion $j\colon F_4\to E_6$ satisfies $j^*(x_i)=x_i$ for $i=4,12,16,24$ and $j^*(x_i)=0$ for $i=10,18$. Then by Corollary \ref{P^1 E_6} and Lemma \ref{P^1 criterion}, $q$ is not an $A_k$-map for $a_{k-1}\le p<a_k$  except for $p=7$ with $k=2$. Let $p=7$. We aim to show that $q$ is not an H-map. As in \cite{Mi2}, there is a homotopy equivalence
  $$E_6\simeq F_4\times S^9\times S^{17}.$$
  Let $\epsilon\colon S^{17}\to E_6$ be the inclusion. Then by \cite{HKMO} the Samelson product $\langle\epsilon,\epsilon\rangle$ is non-trivial. By \cite{To}, $\pi_{34}(S^9)=\pi_{34}(S^{17})=0$, and so $q_*(\langle\epsilon,\epsilon\rangle)\ne 0$. On the other hand, $q_*(\epsilon)=0$. Thus by Lemma \ref{Samelson product}, $q$ is not an H-map.

  Let $(G,H)=(Spin(8),G_2)$ and $j\colon G_2\to Spin(8)$ be the inclusion. Since $BSO(8)\simeq BSpin(8)$ at the odd prime $p$,
  $$H^*(BSpin(8))=\Z/p[p_1,p_2,p_3,e_4].$$
  The cohomology of $BG_2$ is given by
  $$H^*(BG_2)=\Z/p[x_4,x_{12}],\quad|x_i|=i$$
  and so by a degree reason, we may assume
  $$j^*(p_1)=x_4,\quad j^*(p_2)=0,\quad j^*(p_3)=x_{12},\quad j^*(e_4)=0.$$
  Then by Lemmas \ref{SO(8)} and \ref{P^1 criterion}, $q$ is not an $A_k$-map for $a_{k-1}\le p<a_k$.

  (2) Let $(G,H)=(SO(2n),SO(2n-1))$. The inclusion $j\colon SO(2n-1)\to SO(2n)$ satisfies $j^*(p_i)=p_i$ and $j^*(e_n)=0$. Then by Lemmas \ref{P^1 SO(2n)} and \ref{P^1 criterion}, one gets that $q$ is not an $A_k$-map for $a_{k-1}-n+2\le p<a_k-n+2$ unless $p=n$ for $k=2$. Let $p=n$ and suppose $q$ is an H-map. Let $\theta\colon S^{2n-1}\to S^{2n-1}\times SO(2n-1)\simeq SO(2n)$ be the inclusion. In \cite{Ma} it is shown that the Samelson product $\langle\theta,\theta\rangle$ in $SO(2n)$ is non-trivial. Then since $\pi_{4n-2}(S^{2n-1})=0$ as in \cite{To}, $q_*(\langle\theta,\theta\rangle)$ is non-trivial. But since $q_*(\theta)=0$, we obtain a contradiction by Lemma \ref{Samelson product}. Thus $q$ is not an H-map.
\end{proof}

\begin{proof}
  [Proof of Theorem \ref{main}]
  The proof is done by Theorem \ref{main general} and Propositions \ref{a_1} and \ref{a_k} except that $q\colon E_6\to F_4$ is an $A_k$-map for $p=12k-5$. Let $p=12k-5$. From the proofs of Theorem \ref{main general} and Proposition \ref{regular A_k refined}, one can see that it suffices to show that the adjoint of the identity map $\Sigma E_6\to BE_6$ extends to $X\to BE_6$, where $X$ is as in the proof of Proposition \ref{regular A_k refined}. Note that $X=X^{(24k-12)}\cup e^{24k-6}$. Then since $\pi_{2i-1}(BE_6)=0$ for $i\le 12k-6$ and $12k-3$ by \cite{To}, there is an extension $X\to BE_6$ as desired. Thus the proof is complete.
\end{proof}

\begin{proof}
  [Proof of Corollary \ref{C_n}]
  The equivalence of (2) and (3) is proved in \cite{HKT}, and Saumell \cite{Sa} proved that $SU(2n+1)$ is a Williams $C_k$-space if and only if $p>k(2n+1)$. Thus the result follows from Theorem \ref{main}.
\end{proof}

\end{document}